\documentclass{amsart}

\usepackage{mathrsfs}
\usepackage{amsfonts}
\usepackage{amssymb, amsthm, amsmath}
\usepackage{amsxtra}
\usepackage{dsfont}

\newtheorem{theorem}{Theorem}[section]
\newtheorem{lemma}[theorem]{Lemma}
\newtheorem{corollary}{Corollary}

\theoremstyle{definition}
\newtheorem{definition}[theorem]{Definition}

\theoremstyle{remark}

\numberwithin{equation}{section}

\newcounter{thm}

\theoremstyle{plain}
\newtheorem{main_theorem}[thm]{Theorem}

\usepackage[english,polish]{babel}

\newcommand{\pl}[1]{\foreignlanguage{polish}{#1}}

\newcommand{\RR}{\mathbb{R}}
\newcommand{\ZZ}{\mathbb{Z}}

\newcommand{\NN}{\mathbb{N}}

\newcommand{\EE}{\mathbb{E}}

\newcommand{\calF}{\mathcal{F}}
\newcommand{\calM}{\mathcal{M}}

\newcommand{\dnu}{{\: \rm d}\nu}

\newcommand{\ind}[1]{{\mathds{1}_{{#1}}}}

\newcommand{\abs}[1]{{\lvert {#1} \rvert}}

\newcommand{\vnorm}[1]{{\left\lVert {#1} \right\rVert}}

\begin{document}
\selectlanguage{english}

\title[An elementary proof]
{The maximal function and conditional square function control the variation: \\ An elementary proof}

\author{Kevin Hughes}
\address{
		Kevin Hughes\\
		School of Mathematics\\
        The University of Edinburgh\\
	    James Clerk Maxwell Building\\
	    The King's Buildings\\
	    Mayfield Road\\
	    Edinburgh\\
	    EH9 3JZ\\
		UK}
\email{khughes3@staffmail.ed.ac.uk}

\author{Ben Krause}
\address{
		Ben Krause\\
		UCLA Math Sciences Building\\
        Los Angeles\\
	    CA 90095-1555\\
		USA}
\email{benkrause23@math.ucla.edu}

\author{Bartosz Trojan}
\address{
	Bartosz Trojan\\
	Instytut Matematyczny\\
	Uniwersytet \pl{Wroc{\lll}awski}\\
	Pl. Grun\-waldzki 2/4\\
	50-384 \pl{Wroc{\lll}aw}\\
	Poland}
\email{trojan@math.uni.wroc.pl}


\date{October 30, 2014}

\begin{abstract}
	In this note we prove the following good-$\lambda$ inequality, for $r>2$, all $\lambda > 0$,
	$\delta \in \big(0, \frac{1}{2} \big)$
	\[
		C \cdot \nu\big\{ V_r(f) > 3 \lambda ; \calM(f) < \delta \lambda\big\}
		\leq 
		\nu\{s(f) > \delta \lambda\} + \frac{\delta^2}{(r-2)^2} \cdot
		\nu\big\{ V_r(f) > \lambda\big\},
	\]
	where $\calM(f)$ is the martingale maximal function, $s(f)$ is the conditional martingale square
	function, $C > 0$ is (absolute) constant. This immediately proves that $V_r(f)$ is bounded on
	$L^p$, $1 < p <\infty$ and moreover is integrable when the maximal function and the conditional square
	function are.
\end{abstract}

\maketitle

\section{Introduction}
Let $(X, \calF, \nu)$ be $\sigma$-finite measure space with a filtration $\big(\calF_n : n \in \ZZ\big)$, i.e
$\big(\calF_n : n \in \ZZ\big)$ is a sequence of $\sigma$-fields such that
$\calF_n \subseteq \calF_{n+1} \subseteq \calF$. For a martingale $\big(f_n : n \in \ZZ\big)$ we define
the maximal function, the square function
\[
	\calM(f) = \sup_{n \in \ZZ} \abs{f_n}, \quad
	S(f) = \Big(\sum_{n \in \ZZ} \abs{f_n - f_{n-1}}^2\Big)^{1/2},
\]
and a conditional square function
\[
	s(f) = \Big(\sum_{n \in \ZZ} \EE\big[\abs{f_n - f_{n-1}}^2 \big| \calF_{n-1}\big] \Big)^{1/2}.
\]
For a dyadic filtration on $\RR^d$ we have $S(f) = s(f)$ since the square of a martingale difference 
$\abs{f_n - f_{n-1}}^2$ is $\calF_{n-1}$-measurable. It is well known (see \cite[Theorem 9]{burk1} 
and \cite[Theorem 1]{dav}, see also \cite[Theorem 1.1]{bdg}) that for each $p \in [1, \infty)$ there
exists $C_p > 0$ such that
\begin{equation}
	\label{eq:3}
	C_p^{-1} \vnorm{\calM(f)}_{L^p(\nu)} \leq \vnorm{S(f)}_{L^p(\nu)} 
	\leq C_p \vnorm{\calM(f)}_{L^p(\nu)}.
\end{equation}
Also, by convexity Lemma (see \cite[Theorem 3.2]{bdg}) for each $p \in [2, \infty)$ there is $C_p > 0$ satisfying
\[
	\vnorm{s(f)}_{L^p(\nu)} \leq C_p \vnorm{S(f)}_{L^p(\nu)}.
\]
For $p \in (0, 2)$, we have
\[
	\vnorm{\calM f}_{L^p(\nu)} \leq C_p \vnorm{s(f)}_{L^p(\nu)}.
\]
In general, the conditional square function $s$ is not necessary bounded on $L^p(\Omega, \nu)$, for $p \in (1, 2)$. Let
us recall that the filtration $\big(\calF_n : n \in \ZZ\big)$ is regular if there is $R \geq 1$ such that for all
nonnegative martingales $(g_n : n \in \ZZ)$,
\[
	g_n \leq R g_{n-1}.
\]
If the filtration $\big(\calF_n : n \in \ZZ\big)$ is regular then \eqref{eq:3} is valid for all $p \in (0, \infty)$; moreover, for all $p \in (0, \infty)$ there is $C_p > 0$ such
that
\[
	C_p^{-1} \vnorm{s(f)}_{L^p(\nu)} \leq \vnorm{S(f)}_{L^p(\nu)} \leq C_p \vnorm{s(f)}_{L^p(\nu)}.
\]
Another family of operators which measure oscillation are the \emph{$r$-variation} 
operators defined for $r \geq 1$ by
\[
	V_r(f)
	= 
	\sup_{n_0 < n_1 < \ldots < n_J} 
	\Big( \sum_{j = 1}^J 
	\abs{f_{n_j} - f_{n_{j-1}}}^r \Big)^{1/r}.
\]
These variation operators are more difficult to control than the maximal function $\calM$. In fact, for any
$n_0 \in \ZZ$, one may pointwise dominate
\[
	\calM(f) \leq V_r(f) + |f_{n_0}|,
\]
where $r \geq 1$ is arbitrary. We further remark that the variation operators become larger 
(more sensitive to oscillation) as $r$ \emph{decreases}. The fundamental boundedness result
concerning the $r$-variation operators is due to L\'{e}pingle.
\begin{theorem}[\cite{le}]
	For each $p \in [1, \infty)$ there is $A_p > 0$ such that for all $f \in L^p(X, \nu)$
	and $r > 2$
	\[
		\big\lVert V_r(f) \big\rVert_{L^p(\nu)} 
		\leq 
		A_p \frac{r}{r-2} \vnorm{f}_{L^p(\nu)}, \quad (p > 1),
	\]
	and for all $\lambda > 0$
	\[
		\nu\big\{V_r(f) > \lambda \big\}
		\leq
		A_1
		\frac{r}{r-2} 
		\lambda^{-1} 
		\vnorm{f}_{L^1(\nu)}, \quad (p = 1).
	\]
\end{theorem}
We remark that the range of $r$ in the above theorem is sharp, since these estimates can fail
for $r \leq 2$, (see e.g. \cite{jw, q}).

By now, comparatively simple proofs of L\'{e}pingle's theorem can be found in Pisier and Xu
\cite{px} and Bourgain \cite{bou} (see also \cite{jsw}). The idea was to leverage known estimates
for jump inequalities to recover variational estimates. Let us recall that the number of
$\lambda$-jumps, denoted by $N_\lambda(f)$, is equal to the supremum over $J \in \NN$ such that
there is an increasing sequence $n_0 < n_1 < \ldots < n_J$ satisfying
\[
	\abs{f_{n_j} - f_{n_{j-1}}} > \lambda
\]
for all $1 \leq j \leq J$. The key result concerning $\lambda$-jumps is the following theorem.
\begin{theorem}[\cite{bou, px}]
	\label{classic}
	For each $p \in [1, \infty)$ there exist $B_p > 0$ such that for all $f \in L^p(X, \nu)$
	and $\lambda > 0$
	\[ 
		\big\| 
		\lambda N_\lambda^{1/2}(f) 
		\big\|_{L^p(\nu)}
		\leq B_p 
		\|f\|_{L^p(\nu)}, \quad (p > 1),
	\]
	and
	\[ 
		\nu\big\{ 
		\lambda N_\lambda^{1/2}(f) > t
		\big\}
		\leq 
		B_1 t^{-1} \|f\|_{L^1(\nu)}, \quad (p = 1),
	\]
	for any $t > 0$.
\end{theorem}
The goal of this note is to provide a new and elementary proof of L\'{e}pingle's result. 
The significance of our approach is that it sheds new insight into the relationship between maximal
function, conditional square function, and variation operator. Specifically, we prove the following
theorem.
\begin{main_theorem}
	\label{thm:1}
	There is $C > 0$ such that for all $\delta \in \big(0, \frac{1}{2}\big)$, $r > 2$ and
	$\lambda > 0$
	\begin{equation}
		\label{eq:1}
		C \cdot \nu\big\{V_r(f) > 3 \lambda; \calM(f) \leq \delta \lambda\big\}
		\leq
		\nu\big\{s(f) > \delta \lambda \big\}
		+ \frac{\delta^2}{(r-2)^2} \cdot
		\nu\big\{V_r(f) > \lambda \big\}.
	\end{equation}
\end{main_theorem}
In particular, by integrating distribution functions we obtain that for all $p \in (0, \infty)$ and
$r > 2$, $V_r(f) \in L^p(X, \nu)$ whenever $s(f)$ and $\calM(f)$ are in $L^p(X, \nu)$.

\subsection{Acknowledgments}
The authors wish to thank Konrad Kolesko, Christoph Thiele, and Jim Wright for helpful discussions,
and Michael Lacey and Terence Tao for their careful reading and support. Furthermore, this work
was completed during the ``Harmonic Analysis and Partial Differential Equations'' trimester program
at Hausdorff Research Institute for Mathematics; the authors wish to thank HIM for their
hospitality.

\section{The Proof}
We begin with a preliminary lemma.
\begin{lemma}
	\label{weak}
	There is $C > 0$ such that for any $A \in \calF_m$, all $\lambda > 0$, and
	$\delta \in \big(0, \frac{1}{2}\big)$
	\[
		\nu \big\{A; V_r(f) > \lambda; \calM(f) \leq \delta \lambda \big\}
		\leq C \lambda^{-2} (r-2)^{-2} \int_A \abs{f}^2 \dnu
	\]
	for each $f \in L^2(X, \nu)$ satisfying 
	\[
		f_n \cdot \ind{A} = 0
	\]
	for all $n \leq m$.
\end{lemma}
\begin{proof}
	By homogeneity, it suffices to prove the result with $\lambda = 1$. We can pointwise dominate 
	the variation as in \cite[\S 3]{bou}
	\[ 
		\big(V_r(f)\big)^r \leq \sum_{l \in \ZZ} 2^{rl} N_{2^l}(f).
	\]
	Let $s = (r+2)/2$. Since $\calM(f) < \delta < 1/2$, the above sum runs over $l \leq 0$, which
	leads to the containment
	\[
		\big\{A; V_r(f) > 1; \calM(f) < \delta \big\} 
		\subset
		\Big\{A; \sum_{l \leq 0} 2^{rl} N_{2^l}(f) > 1 \Big\}
		\subset
		\bigcup_{l \leq 0}
		\left\{ 2^{sl} N_{2^l}(g) > c_r \right\},
	\]
	where $g = f \cdot \ind{A}$ and
	\[
		c_r^{-1} = \frac{1}{2} \sum_{l \leq 0} 2^{(r-2) l/2}.
	\]
	Let us observe that $c_r = o(r-2)$. In light of Theorem \ref{classic}, this immediately
	leads to the majorization
	\footnote{We write $X \lesssim Y$, or $Y \gtrsim X$ to denote the estimate $X \leq C Y$ for an absolute
	constant $C > 0$.}
	\[
		\nu\big\{A; V_r(f) > 1; \calM(f) < \delta \big\} 
		\lesssim
		c_r^{-1}
		\sum_{l \leq 0} 2^{(s-2)l} \int_A \abs{f}^2 \dnu
		\lesssim (r-2)^{-2} \int_A \abs{f}^2 \dnu.
	\] 
\end{proof}

\begin{proof}[Proof of Theorem \ref{thm:1}] 
	By homogeneity, it will suffice to prove \eqref{eq:1} for $\lambda = 1$.

	Let $B = \{s(f) > \delta\}$, $B^* = \big\{\calM \big(\ind{B}\big) > 1/2 \big\}$ and 
	$G = \big(B^*\big)^c$. By Doob's inequality, we have
	\[
		\nu\big(B^*\big) \lesssim \int \abs{\ind{B}}^2\dnu = \nu\{s(f) > \delta\}.
	\]
	Therefore, it is enough to show that there is $C > 0$ such that for all $\delta \in (0, \frac{1}{2})$ and
	any $N \in \NN$
	\[
		\nu\big\{V_r(f_n : -N \leq n ) > 3; \calM(f) < \delta; G\big\}
		\leq
		C \delta^2 \cdot \nu\{V_r(f) > 1\}.
	\]
	Let $\sigma$ be a stopping time defined to be equal to the minimal $m \in \ZZ$ such that
	\[
		V_r\big(f_n : -N \leq n \leq m\big) > 1.
	\]
	Notice, that on the set $\big\{V_r(f_n : -N \leq n) > 3\big\}$, we have $-N \leq \sigma < \infty$.
	Next,
	\[
		V_r(f_n : -N \leq n) \leq V_r\big(f_n - f_{n \land \sigma} : -N \leq n \big) + 2 \calM(f)
		+ V_r\big(f_{n \land \sigma} : -N \leq n < \sigma\big),
	\]
	thus for $g = f - f_\sigma$ we have
	\[
		\big\{V_r(f_n : -N \leq n) > 3; \calM(f) < \delta; G\big\}
		\subseteq
		\big\{V_r(g) > 1; \calM(g) < 2 \delta; G\big\}.
	\]
	We are going to prove that for each $m \in \ZZ$
	\[
		\nu\big\{
			V_r(g) > 1; \calM(g) < 2 \delta; G; \sigma = m
		\big\}
		\lesssim
		\delta^2 \cdot
		\nu\{\sigma = m\}.
	\]
	For $n \in \ZZ$ we define $U_n =\{x: \EE[\ind{B} | \calF_n](x) < 1/2\}$. We notice that,
	if $x \in G$ then $x \in U_n$ for all $n \in \ZZ$. Let
	\[
		\tilde{g}(x) = \sum_{n \in \ZZ} \big(g_n(x) - g_{n-1}(x)\big) \cdot \ind{U_{n-1}}(x).
	\]
	We observe that $g_k(x) = \tilde{g}_k(x)$ for all $x \in G$ and $k \in \ZZ$. Indeed,
	$\big(g_n - g_{n-1}\big) \cdot \ind{U_{n-1}}$ is $\calF_n$-measurable and
	\[
		\EE\big[ \big(g_n - g_{n-1}\big) \cdot \ind{U_{n-1}} \big| \calF_{n-1} \big] 
		= 0.
	\]
	Thus for $x \in G$ we have
	\[
		\tilde{g}_k(x)
		=
		\sum_{n \leq k} \big(g_n(x) - g_{n-1}(x)\big) \cdot \ind{U_{n-1}}(x)
		=g_k(x).
	\]
	Therefore, we obtain 
	\begin{align*}
		\nu\big\{V_r(g) > 1; \calM(g) < 2\delta; G; \sigma = m\big\}
		&=
		\nu\big\{V_r(\tilde{g}) > 1; 
		\calM (\tilde{g}) < 2\delta; G; \sigma = m\big\}\\
		&\leq
		\nu\big\{V_r(\tilde{g}) > 1; \calM (\tilde{g}) < 2\delta; \sigma = m\big\}.
	\end{align*}
	Because $\tilde{g} = 0$ on the set $\{\sigma = m\}$, by Lemma \ref{weak} we conclude
	\[
		\nu\big\{V_r(\tilde{g}) > 1; \calM(\tilde{g}) < 2\delta; \sigma = m\big\}
		\lesssim
		(r-2)^{-2}
		\int_{\{\sigma = m\}}
		\abs{\tilde{g}}^2 \dnu.
	\]
	Next, since $s$ preserves $L^2$-norm thus
	\[
		\int_{\{\sigma = m \}}
		\abs{\tilde{g}}^2 \dnu
		=
		\int
		s\big(\tilde{g} \cdot \ind{\{\sigma = m\}}\big)^2 \dnu
		=
		\sum_{n \in \ZZ}
		\int_{\{\sigma = m\}}
		\EE\big[\abs{g_n - g_{n-1}}^2 \big| \calF_{n-1} \big] \cdot \ind{U_{n-1}} \dnu.
	\]
	Since $\ind{U_{n-1}} \leq 2 \cdot \EE[\ind{B^c} | \calF_{n-1}]$ we get
	\begin{align*}
		\int_{\{\sigma = m\}} \abs{\tilde{g}}^2 \dnu
		& \leq
		2
		\sum_{n \in \ZZ}
		\int_{\{\sigma = m\}} 
		\EE\big[\abs{g_n - g_{n-1}}^2 \big| \calF_{n-1}\big]
		\cdot
		\EE[\ind{B^c} | \calF_{n-1}]  \dnu \\
		& =
		2
		\int_{\{\sigma = m\}}
		s(f)^2 \cdot \ind{B^c} \dnu
	\end{align*}
	which is bounded by $2 \delta^2 \cdot \nu\{\sigma = m\}$.
\end{proof}

\section{Applications to dyadic $A_\infty$-weights}
We remark that in the case of the dyadic filtration on $\RR^d$, the proof generalizes to handle 
measures given by $w$, dyadic $A_\infty$-weights. First, let us recall the following definition.
\begin{definition}
	A non-negative locally integrable function $w$ belongs to \emph{dyadic $A_\infty$}, 
	if for every $\epsilon > 0$ there exists $\gamma >  0$ so that for every dyadic interval $I$
	and any measurable set $E \subset I$, if $\abs{E} \leq \gamma \cdot \abs{I}$ then
	\begin{equation}
		\label{eq:2}
		w(E) \leq \epsilon w(I).
	\end{equation}
	If additionally, there is $C > 0$ such that for all dyadic intervals $I$
	\[
		C^{-1} w(I_l) \leq w(I_r) \leq C w(I_l),
	\]
	where $I_l$ and $I_r$ are, respectively, left and right children of $I$, then $w$ is called
	dyadic doubling.
\end{definition}
\begin{corollary}
	Let $w$ be a dyadic $A_\infty$-weight. There exist $C > 0$ so that for each $\epsilon > 0$
	and $r > 2$ there is $\delta > 0$ such that for all $\lambda > 0$ 
	\[
		w\big\{V_r(f) > 3\lambda; \calM(f) < \delta\lambda\big\}
		\leq
		C \cdot w\{S(f) > \delta \lambda\} + \epsilon \cdot w\{V_r(f) > \lambda\}.
	\]
\end{corollary}
\begin{proof}
	Using the notation as in the proof of Theorem \ref{thm:1} we may write
	\[
		\big\{V_r(f) > 3; \calM(f) < \delta; G; \sigma = m \big\}
		\subseteq 
		\{\sigma = m\}
	\]
	and
	\[
		\big\lvert \big\{
			V_r(f) > 3; \calM(f) < \delta; G; \sigma = m
		\big\} \big\rvert
		\leq
		C \frac{\delta^2}{(r-2)^2} \cdot
		\big\lvert \big\{
		\sigma = m 
		\big\} \big\rvert.
	\]
	Given $\epsilon > 0$ we take $\delta > 0$ small enough so that
	$C \frac{\delta^2}{(r-2)^2} \leq \gamma$. Then, by \eqref{eq:2} we get
	\[
		w\big\{V_r(f) > 3; \calM(f) < \delta; G ; \sigma = m \big\}
		\leq
		\epsilon \cdot w\big\{\sigma = m \big\}.
	\]
	Since for the dyadic filtration $S(f) = s(f)$ we conclude the proof.
\end{proof}

Again, by integrating distribution functions for each $p \in [1, \infty)$ and $r > 2$ we can find 
$C_{p,r} > 0$ such that
\[
	C_{p,r} \vnorm{V_r(f)}_{L^p(w)} \leq \vnorm{S(f)}_{L^p(w)} + \vnorm{\calM(f)}_{L^p(w)}.
\]
By \cite[\S 2]{wil}, for each $w$, a dyadic $A_\infty$-weight there is $C_p > 0$
\[
	\vnorm{\calM(f)}_{L^p(w)} \leq C_p \vnorm{S (f)}_{L^p(w)},
\]
so the square function alone dominates $V_r$ in $L^p(w)$. In the case where $w$ is a dyadic
doubling, we have the reverse inequality as well 
\[
	\vnorm{S(f)}_{L^p(w)} \leq C_p^{-1} \vnorm{\calM(f)}_{L^p(w)},
\]
and thus the maximal function alone dominates $V_r$ in $L^p(w)$.

\bibliographystyle{amsplain}

\end{document}